\documentclass[a4paper,12pt]{amsart}
\usepackage[cp1251]{inputenc}
\usepackage{amssymb,upref}
\usepackage{graphicx}
\DeclareGraphicsExtensions{.pdf,.png,.jpg}
\usepackage{chngcntr}
\usepackage{amsthm}
\usepackage[T1,OT1]{fontenc}
\usepackage[hidelinks]{hyperref}
\newtheorem{theoremalpha}{Theorem}
\newcounter{example} 

\newcounter{alphcount}
\newenvironment{alphthm}
  {\stepcounter{alphcount}\begin{theoremalpha}}
  {\end{theoremalpha}}
\theoremstyle{plain}
\newtheorem{lemma}{Lemma}
\newtheorem{corollary}{Corollary}

\theoremstyle{definition}

\theoremstyle{remark}

\uccode`ґ=`Ґ \uccode`Ґ=`Ґ \lccode`Ґ=`ґ \lccode`ґ=`ґ %
\uccode`є=`Є \uccode`Є=`Є \lccode`Є=`є \lccode`є=`є %
\uccode`і=`І \uccode`І=`І \lccode`І=`і \lccode`і=`і %
\uccode`ї=`Ї \uccode`Ї=`Ї \lccode`Ї=`ї \lccode`ї=`ї %

\begin{document}

\title[Properties of one family of self-similar sets]{Topological, metric and fractal properties of one family of self-similar sets}
\author[D. Karvatskyi]{{D. Karvatskyi}}
\address{Dmytro Karvatskyi, Institute of Mathematics of NAS of Ukraine, 3, Tereshchenkivska Street, Kyiv, 01030, Ukraine}
\email{karvatsky@imath.kiev.ua}
\thanks{The author worked at the University of St Andrews
within the framework of the Isaac Newton Institute Solidarity Programme and also received
additional support from the London Mathematical Society.}

\begin{abstract}
Depending on a natural parameter $l$, we study the topological, metric, and fractal properties of the homogeneous self-similar set
$$K_{l}=\left\{\sum_{i=1}^{\infty} \frac{\varepsilon_i}{(2l+2)^i} : (\varepsilon_i) \in \{0, 2, 4, \dots, 2l, 2l+1, 2l+3, \dots, 4l+1 \}^{\mathbb{N}} \right\}.$$ In particular, we prove that $K_l$ is a Cantorval, a perfect set on the real line with a non-empty interior and fractal boundary. Additionally, we show that $K_l$ has Lebesgue measure $1$ and that its boundary has Hausdorff dimension $\log(2l+1)/\log(2l+2)$.
\end{abstract}

\maketitle

\section{Some results from series theory}

Firstly, we recall some results related to the theory of numerical series. Let us consider a convergent positive series with non-increasing terms
$$\sum_{n=1}^{\infty}{u_n}<+\infty, \ \ \ 0<u_{n+1} \leq u_{n}, \ \ \ \forall n \in \mathbb{N}.$$
In a standard way, we denote the $n$-th tail of the series by 
$$U_n:=\sum^{\infty}_{i=n+1}u_i.$$
The \textbf{set of subsums} of $\sum u_n$, also known as the \textbf{achievement set} of the sequence $(u_n)$, is defined as
\begin{equation*}
 \label{incomplit sum}
E(u_n)=\left\{\sum^{\infty }_{n=1}{{\varepsilon }_nu_n}: ~ ~ ~ ({\varepsilon }_n) \in \{0, 1\}^{\mathbb{N}} \right\}.
\end{equation*}

The study of sets of subsums of numerical series was initiated in \cite{kakeya}, where it was established that $E(u_n) \subseteq [0, \sum{u_n}]$ is a perfect set that is symmetric with respect to its midpoint. This object was later rediscovered in \cite{Hornich} and \cite{Menon}. It turns out that the relation between terms and tails of the series plays a crucial role in determining the topological nature of $E(u_n)$.  As a consequence of these early papers, we have the following result.

\begin{alphthm}
\label{KHM}
The set of subsums of a convergent positive series $\sum u_n$ with non-increasing terms is:
\begin{enumerate}  
\item\label{thm:KHM.2} a finite union of bounded closed intervals if and only if the inequality $u_{n} \leq U_{n}$
holds for all but finitely many $n$ (a closed bounded interval if and only if $u_n\leq U_n$ for all $n \in \mathbb{N}$);
\item\label{thm:KHM.4} homeomorphic to the classical Cantor set if the inequality $u_{n} > U_{n}$ holds for all but finitely many $n$.
\end{enumerate}
\end{alphthm}
For simplicity, we say that a series $\sum {u_n}$ satisfies \textbf{Kakeya's condition} if either $u_n\leq U_n$ for all sufficiently large $n$, or $u_n>U_n$ for all sufficiently large $n$. For a long time, it was believed that the set of subsums of an arbitrary positive convergent series is either a finite union of intervals or a Cantor set. However, it was later discovered that a third mixed type also exists.
In \cite{NS}, a full topological classification of the set of subsums of an arbitrary positive convergent series was given.

\begin{alphthm}
\label{Classification}
The set of subsums of a convergent positive series is either:
\begin{enumerate}
 \item a finite union of closed intervals;
 \item homeomorphic to the Cantor set (or a Cantor set for short);
 \item an M-Cantorval, i.e., a set that is homeomorphic to
 $$C \cup \bigcup_{n=1}^{\infty} G_{2n-1} \equiv [0, 1] \setminus \bigcup_{n=1}^{\infty} G_{2n},$$
where $C$ is the Cantor ternary set and $G_n$ is the union of the $2^{n-1}$ open middle thirds, which are removed from $[0, 1]$ at the $n$-th step in the construction of $C$.
\end{enumerate}
\end{alphthm} 

\begin{corollary}
\label{non-Kakeya}
If the series $\sum {u_n}$ does not satisfy Kakeya's condition, then $E(u_n)$ is either a Cantor set or a Cantorval.
\end{corollary}

A large number of articles are devoted to the search for necessary and sufficient conditions for the set of subsums to be a Cantorval. Despite essential progress for some series, the problem is quite difficult in general frameworks. The most significant results in this direction were obtained for multigeometric series  
$$k_1+k_2+\dots+k_m+k_1q+\dots+k_mq+\dots + k_1q^i+\dots+k_mq^i+\dots,$$
where $k_1, k_2, \dots, k_m$ are fixed positive scalars, and $0<q<1$. In particular, conditions for the set of subsums of multigeometric series to be a Cantorval were established in \cite{Banakh}, \cite{BP17}, \cite{Bartoszewicz}, \cite{Ferdinands}, \cite{Jones}, \cite{JK}, \cite{Nit15} with a trigonometric extension in \cite{PK23} and a more general functional setting in \cite{KMV}.
Some conditions for the set of subsums of a non-multigeometric series to be a Cantorval were found in \cite{VMPS19}. 
We also refer to \cite{GP25}, where numerous open problems related to the topic are discussed.

\section{Self-similar sets on the real line}

Let us now recall some facts from the theory of self-similar sets \cite{BSS} that will be used later on. Let us consider an \textbf{iterated function system (IFS)} that is a finite collection of contractive similarities on the real line $\mathbb{R}$:
$$\Phi=\{ f_i(x)=r_i \cdot x + t_i\}_{i=1}^{N},$$
where $0<r_i<1$ and $t_i \in \mathbb{R}$ for all $1 \leq i \leq N$. In case $r_i=r$ for all $1 \leq i \leq N$, the IFS is called \textbf{homogeneous} or \textbf{equicontractive}.

According to the classical result of \cite{Hutchinson}, for any $\Phi$, there exists a unique non-empty compact set $K \subset \mathbb{R}$ such that
$$K=\bigcup_{i=1}^{N} f_i(K),$$
which is called the \textbf{attractor} of $\Phi$ or the \textbf{self-similar set} associated with $\Phi$. A self-similar set consists of a finite union of smaller copies of itself that may overlap or be disjoint.

An IFS $\Phi=\{ f_i(x)\}$ is said to satisfy the \textbf{open set condition} (OSC) if there exists a non-empty bounded open set $V$ such that:
\begin{enumerate}
    \item the sets $f_i(V)$ for \( i = 1, 2, \ldots, N \) are pairwise disjoint;
    \item \( f_i(V) \subseteq V \) for all \( i \).
\end{enumerate}
Fulfillment of the OSC ensures that the self-similar set generated by $\Phi$ consists of self-similar copies that are pairwise disjoint or only touch each other.

A non-negative number $s$ satisfying the equation
$$r_1^s+r_2^s+\dots +r_N^s=1$$
is called the \textbf{similarity dimension} of the IFS $\Phi$. The similarity dimension serves as a natural upper bound for the Hausdorff dimension of $K$. Therefore, if for a self-similar set on the real line $s<1$, then it is a totally disconnected set -- its only connected components are singletons.
Due to Hutchinson \cite{Hutchinson}, if an IFS satisfies the OSC, then
$$0<H^{s}(K)<\infty,$$
where $H^{s}(K)$ denotes the $s$-dimensional Hausdorff measure. As a consequence, the similarity dimension and the Hausdorff dimension of $K$ coincide, i.e., $\dim_{H} K=s.$

The following important result from \cite{Schief} establishes a connection between the Lebesgue measure and the topological structure of the attractor with integer similarity dimension.

\begin{alphthm}
\label{T:Schief}
Let $\Phi$ be an IFS on $\mathbb{R}$ with attractor $K$ and similarity dimension $1$. If $K$ has positive Lebesgue measure, then $K$ has a nonempty interior and $\Phi$ satisfies the OSC.
\end{alphthm}

\section{One family of Cantorvals}

In this section, we compute the Lebesgue measure and the Hausdorff dimension of the boundary of Cantorvals belonging to a certain countable one-parameter family. To date, the metric and fractal properties of Cantorvals have been established only for a very limited number of examples. In particular, the Lebesgue measure of Cantorvals from another family was calculated in \cite{Banakiewicz}, while the Hausdorff dimension of their boundaries was obtained in \cite{PKM}.

Depending on the parameter $l \in \mathbb{N}$, we consider a homogeneous self-similar set
$$K_{l}=\left\{\sum_{i=1}^{\infty} \frac{\varepsilon_i}{(2l+2)^i} : (\varepsilon_i) \in T_{l}^{\mathbb{N}} = \{0, 2, 4, \dots, 2l, 2l+1, 2l+3, \dots, 4l+1 \}^{\mathbb{N}} \right\}$$
generated by an IFS $\Psi_l=\left\{ \psi_i(x)\right\}_{i=1}^{2l+2}$ defined as
\begin{enumerate}
\item $\psi_i(x)=\frac{x+(2i-2)}{2l+2}$ for $1 \leq i \leq l+1$;
\item $\psi_i(x)=\frac{x+(2i-3)}{2l+2}$ for $l+2 \leq i \leq 2l+2$.
\end{enumerate}
On the other hand, $K_l$ is the set of subsums of a multigeometric series
$$\sum_{n=1}^{\infty}z_n=\frac{2l+1}{2l+2}+\underbrace{\frac{2}{2l+2}+\dots+\frac{2}{2l+2}}_l+\dots+\frac{2l+1}{(2l+2)^i}+\underbrace{\frac{2}{(2l+2)^i}+\dots+\frac{2}{(2l+2)^i}}_l+\dots$$
In order to reveal the structure of $K_l$, we rely on its dual nature: it is both self-similar and achievable -- it is the set of subsums of $\sum z_n$. Thus, we can leverage some results described in Sections 1 and 2.

\begin{lemma}
\label{Basic properties}
The following properties of $\Psi_l$ and its attractor are easy to verify:
\begin{enumerate}

\item $H=\left[ 0, \frac{4l+1}{2l+1}\right]$ serves as the convex hull of $K_l$;

\item $K_l$ is symmetric with respect to its midpoint;

\item $\min \psi_i(H)=\frac{2i-2}{2l+2}$ for $1 \leq i \leq l+1$;\\
$\min \psi_i(H)=\frac{2i-3}{2l+2}$ for $l+2 \leq i \leq 2l+2$;

\item $\max \psi_i(H)=\frac{4il+2i-1}{(2l+1)(2l+2)}$ for $1 \leq i \leq l+1$;\\
$\max \psi_i(H)=\frac{4il-2l+2i-2}{(2l+1)(2l+2)}$ for $l+2 \leq i \leq 2l+2$.

\end{enumerate}
\end{lemma}

\begin{lemma}
\label{LM1}
The interval $I=\left[ \frac{2l}{2l+1}, 1\right]$ is entirely contained in $K_l$.
\end{lemma}
\begin{proof}
It is well known that every $a \in [0, 1]$ can be represented as
\begin{equation}
\label{Delta-expansion}
a=\Delta_{a_1 a_2 \dots a_n \dots}^{2l+2}=\sum_{n=1}^{\infty} \frac{a_n}{(2l+2)^n},
\end{equation}
where $(a_i) \in \{ 0, 1, 2, \dots, 2l+1\}^{\mathbb{N}}=A^{\mathbb{N}}$. The representation \eqref{Delta-expansion} is commonly referred to as the $(2l+2)$-expansion of $a$.

We first establish that the set $K_l$ is dense in the interval \( \left[\frac{2l+1}{2l+2}, 1\right] \), i.e.,  
$$
(\forall \varepsilon > 0) \quad \left(\forall a \in \left[\frac{2l+1}{2l+2}, 1\right]\right) \quad (\exists b \in K_l) \ \ \ \text{such that} \ \ \ \quad |a - b| < \varepsilon.
$$  
Observe that for any \( a \in \left[\frac{2l+1}{2l+2}, 1\right] \), there exists an infinite sequence \( (a_i) \in A^{\mathbb{N}} \) with \( a_1 = 2l+1 \) such that  
\[
a = \Delta_{a_1 a_2 \dots a_n \dots}^{2l+2} = \sum_{n=1}^{\infty} \frac{a_n}{(2l+2)^n}.
\]
In turn, for any $\varepsilon > 0$ there exists $k \in \mathbb{N}$ such that $|\Delta^{2l+2}_{a_1 \dots a_k \dots} - \Delta_{a_1 \dots a_k}^{2l+2}|<\varepsilon$. Indeed, since

$$\Delta^{2l+2}_{a_1 \dots a_k \dots} - \Delta_{a_1 \dots a_k}^{2l+2}=\sum_{n=k+1}^{\infty} \frac{a_{n}}{(2l+2)^{n}} \leq \sum_{n=k+1}^{\infty} \frac{2l+1}{(2l+2)^{n}}=\frac{1}{(2l+2)^{k}},$$
we can always take $k>\lceil -\log{\varepsilon} / \log{(2l+2)}\rceil$.

Next, we show that for every finite sequence $(a_i)_{i=1}^{k} \in A^k$ with $a_1=2l+1$ there exists $(b_i)_{i=1}^{k} \in T_l^{k}$ such that

\begin{equation}
\label{Identity}
\sum_{i=1}^{k} \frac{a_i}{(2l+2)^i}=\sum_{i=1}^{k} \frac{b_i}{(2l+2)^i} \in K_l.
\end{equation}

If $k=1$, take $b_1=a_1=2l+1$. Now let $k\geq 2$ and define $(b_i)_{i=1}^{k} \in T_l^{k}$ using the following algorithm:
\begin{enumerate}
\item Initialization with the last digit:
\[
    b_k =
    \begin{cases}
    a_k, & \text{if } a_k \in T_l, \\
    a_k + 2l + 2, & \text{if } a_k \notin T_l.
    \end{cases}
    \]

\item Recursive assignment for $2 \leq i \leq k-1$ 
\[
    b_i =
    \begin{cases}
    a_i, & \text{if } a_i \in T_l \text{ and } b_{i+1} \in \{ a_{i+1}-1, a_{i+1}\}, \\
    a_i + 2l + 2, & \text{if } a_i \notin T_l \text{ and } b_{i+1} \in \{ a_{i+1}-1, a_{i+1}\}, \\
    2l+1, & \text{if } a_i =0 \text{ and } b_{i+1} \not \in \{ a_{i+1}-1, a_{i+1}\}, \\
    a_i - 1, & \text{if } a_i \in A \setminus T_l \text{ and } b_{i+1} \not \in \{ a_{i+1}-1, a_{i+1}\}, \\
    a_i+2l+1, & \text{if } a_i \in T_l \setminus \{0, 2l + 1\} \text{ and } b_{i+1} \not \in \{ a_{i+1}-1, a_{i+1}\}, \\
    2l, & \text{if } a_i = 2l+1 \text{ and } b_{i+1} \not \in \{ a_{i+1}-1, a_{i+1}\}. \\
    \end{cases}
\]

\item Finally, since $a_1=2l+1$, the first digit of the desired representation can be found as

\[
    b_1 =
    \begin{cases}
    2l + 1, & \text{if } b_{2} \in \{ a_{2}-1, a_{2}\},\\
    2l, & \text{if } b_{2} \not \in \{ a_{2}-1, a_{2}\}.
    \end{cases}
\]

\end{enumerate}
Taking into account that for an arbitrary $a_i \in A \setminus T_l$ 
$$ (a_i+2l+2) \in T_l, \ \ \ \frac{a_i +2l+2}{(2l+2)^i}=\frac{1}{(2l+2)^{i-1}}+\frac{a_i}{(2l+2)^i}$$
we see that the algorithm processes the digits from right to left: each carry is either absorbed or passed to the preceding digit. The identity above shows that these transfers preserve the sum, while the choice of $b_1$ absorbs the final carry. Hence \eqref{Identity} holds. Since $K_l$ is compact, its density in $\left[ \frac{2l+1}{2l+2}, 1\right]$ implies that this interval is totally contained in $K_l$.

Let us recall that the set $K_l$ is symmetric with respect to its midpoint $\frac{4l+1}{4l+2}$, as is an arbitrary set of subsums of a convergent positive series.
From the inclusion
$$\left[ \frac{4l+1}{4l+2}, 1\right] \subset \left[ \frac{2l+1}{2l+2}, 1\right] \subset K_l$$
it follows that
$$\left[ \frac{4l+1}{4l+2} - \left( 1 - \frac{4l+1}{4l+2}\right), 1\right]=\left[ \frac{2l}{2l+1}, 1\right] \subset K_l.$$
\end{proof}

\begin{corollary}
The set $K_l$ is a Cantorval.
\end{corollary}

\begin{proof}
By definition, $K_l$ coincides with the set of subsums of the series $\sum_{n=1}^{\infty}z_n$ with $0<z_{n+1} \leq z_{n}, n \in \mathbb{N}$. Let $Z_n:=\sum_{j=n+1}^{\infty}z_j$ be the $n$-th tail of this series. We observe that 
$$\frac{2}{(2l+2)^i}=z_{i(l+1)} > Z_{i(l+1)}=\frac{4l+1}{(2l+1)(2l+2)^i},$$
while $z_n < Z_n$ for every $n$ not divisible by $l+1$.
Thus, $\sum_{n=1}^{\infty}z_n$ does not satisfy Kakeya's condition, and hence $K_l$ is either a Cantor set or a Cantorval. Since the interval $I$ is contained in $K_l$, only the latter possibility remains.
\end{proof}

\begin{lemma}
\label{LM3}
$K_l \setminus I$ is a countable union of pairwise disjoint affine copies of $K_l$. Moreover, this union contains $2l$ isometric copies of $K_l/(2l+2)^n$ for each $n \in \mathbb{N}$. 
\end{lemma}

\begin{proof}
Since $K_l$ is the attractor of $\Psi_l$, we have
$$K_l=\bigcup_{i=1}^{2l+2} \psi_i(K_l) \subseteq H.$$
Taking into account the basic properties of $\Psi_l$ from Lemma \ref{Basic properties}, we get
$$\psi_i(H) \cap I = \emptyset ~ ~ ~ \text{for all} ~ ~ ~ i \in \{1, 2, \dots, l,l+3, \dotsm, 2l+2\},$$
$$\psi_i(H) \cap \psi_j(H) = \emptyset ~ ~ ~ \text{for all} ~ ~ ~ i \in \{1, 2, \dots, l,l+3, \dotsm, 2l+2\}, 1 \leq j \leq 2l+2, i \neq j.$$
It implies that $\psi_i(K_l)$, for $i \in \{1, 2, \dots, l,l+3, \dots, 2l+2\}$, are pairwise disjoint affine copies of $K_l$ with similarity ratio $1/(2l+2)$ that are totally contained in $K_l \setminus I$. At the same time,
$$\emptyset \neq \psi_{l+1}(H) \cap \psi_{l+2}(H) \subset I, \ \ \ \text{but} \ \ \ \psi_{l+1}(H) \nsubseteq I, \ \ \ \psi_{l+2}(H) \nsubseteq I.$$

Let us examine $\psi_{l+1}(K_l)$ and $\psi_{l+2}(K_l)$ in more detail. Taking into account the self-similarity of $K_l$, we can write
$$\psi_{l+1}(K_l)=\bigcup_{i=1}^{2l+2} \psi_{l+1} \circ \psi_{i}(K_l), \ \ \ \psi_{l+2}(K_l)=\bigcup_{i=1}^{2l+2} \psi_{l+2} \circ \psi_{i}(K_l).$$
For $\psi_{l+1}(K_l)$, by using Lemma \ref{Basic properties}, we get
$$\psi_{l+1} \circ \psi_{i} (H) \subset I \ \ \ \text{for} \ \ \ l+2 \leq i \leq 2l+2,$$
$$\psi_{l+1} \circ \psi_{i} (H) \cap I = \emptyset \ \ \ \text{for} \ \ \ 1 \leq i \leq l,$$
$$\psi_{l+1} \circ \psi_{l+1} (H) \cap I \neq \emptyset, \ \ \ \text{but} \ \ \ \psi_{l+1} \circ \psi_{l+1} (H) \not\subseteq I,$$
$$\psi_{l+1} \circ \psi_{i} (H) \cap \psi_{l+1} \circ \psi_{j} (H) = \emptyset \ \ \  \text{for} \ \ \ 1 \leq i \leq l, 1 \leq j \leq 2l+2, i \neq j.$$
A symmetric situation holds for $\psi_{l+2}(K_l)$, where
$$\psi_{l+2} \circ \psi_{i} (H) \subset I \ \ \ \text{for} \ \ \ 1 \leq i \leq l+1,$$
$$\psi_{l+2} \circ \psi_{i} (H) \cap I = \emptyset \ \ \ \text{for} \ \ \ l+3 \leq i \leq 2l+2,$$
$$\psi_{l+2} \circ \psi_{l+2} (H) \cap I \neq \emptyset, \ \ \ \text{but} \ \ \ \psi_{l+2} \circ \psi_{l+2} (H) \not\subseteq I,$$
$$\psi_{l+2} \circ \psi_{i} (H) \cap \psi_{l+2} \circ \psi_{j} (H) = \emptyset \ \ \  \text{for} \ \ \ l+3 \leq i \leq 2l+2, 1 \leq j \leq 2l+2, i \neq j.$$
It means that $\psi_{l+1} \circ \psi_i (K_l)$ for $1 \leq i \leq l$, and $\psi_{l+2} \circ \psi_j (K_l)$ for $l+3 \leq j \leq 2l+2$, are pairwise disjoint affine copies of $K_l$ with similarity ratio $1/(2l+2)^2$ that are totally contained in $K_l \setminus I$.

By iterating this process for $n$, we can show that
$$\underbrace{\psi_{l+1} \circ \dots \circ \psi_{l+1}}_n \circ \psi_{i} (H) \subset I \ \ \ \text{for} \ \ \  l+2 \leq i \leq 2l+2,$$
$$\underbrace{\psi_{l+1} \circ \dots \circ \psi_{l+1}}_n \circ \psi_{i} (H) \cap \underbrace{\psi_{l+1} \circ \dots \circ \psi_{l+1}}_n \circ \psi_{j} (H) = \emptyset \ \ \  \text{for} \ \ \ 1 \leq i \leq l, i \neq j.$$
$$\underbrace{\psi_{l+1} \circ \dots \circ \psi_{l+1}(H)}_{n+1} \cap I \neq \emptyset, \ \ \ \text{but} \ \ \  \underbrace{\psi_{l+1} \circ \dots \circ \psi_{l+1}(H)}_{n+1} \nsubseteq I,$$
and
$$\underbrace{\psi_{l+2} \circ \dots \circ \psi_{l+2}}_n \circ \psi_{i} (H) \subset I \ \ \ \text{for} \ \ \ 1 \leq i \leq l+1,$$
$$\underbrace{\psi_{l+2} \circ \dots \circ \psi_{l+2}}_n \circ \psi_{i} (H) \cap \underbrace{\psi_{l+2} \circ \dots \circ \psi_{l+2}}_n \circ \psi_{j} (H) = \emptyset \ \ \  \text{for} \ \ \ l+3 \leq i \leq 2l+2, i \neq j.$$
$$\underbrace{\psi_{l+2} \circ \dots \circ \psi_{l+2}(H)}_{n+1} \cap I \neq \emptyset, \ \ \ \text{but} \ \ \  \underbrace{\psi_{l+2} \circ \dots \circ \psi_{l+2}(H)}_{n+1} \nsubseteq I.$$
Moreover, we observe that for an arbitrary $x \in K_l$, $i\in\{1,\dots,l\}$ and $j\in\{l+3,\dots,2l+2\}$
$${\underbrace{\psi_{l+1} \circ \dots \circ \psi_{l+1}}_n \circ \psi_{i} (x) \uparrow \frac{2l}{2l+1}}\ \ \ \text{and} \ \ \ {\underbrace{\psi_{l+2} \circ \dots \circ \psi_{l+2}}_n \circ \psi_{j} (x) \downarrow 1} \ \ \ \text{as} \ \ \ n \rightarrow \infty.$$

Since $\psi_{l+1}$ and $\psi_{l+2}$ are contractions with fixed points $\frac{2l}{2l+1}$ and $1$, respectively, the nested copies obtained by repeatedly applying either map shrink to the corresponding fixed point in $I$. Thus every point of $K_l\setminus I$ belongs to an outer copy after finitely many steps. Copies obtained at different steps are also disjoint: each outer copy is disjoint from the nested copy containing all later copies on the same side of $I$, while the left and right copies are separated by $I$. This means that $K_l \setminus \left[ \frac{2l}{2l+1}, 1\right]$ consists of a countable union of pairwise disjoint affine copies of $K_l$ having the form
$$\underbrace{\psi_{l+1} \circ \dots \circ \psi_{l+1}}_n \circ \psi_{i} (K_l), \ \ \ \text{for each} \ \ \ n \in \mathbb{N}_0, \ i \in \{1, \dots, l \},$$
and
$$\underbrace{\psi_{l+2} \circ \dots \circ \psi_{l+2}}_n \circ \psi_{j} (K_l), \ \ \ \text{for each} \ \ \ n \in \mathbb{N}_0, \ j \in \{l+3, \dots, 2l+2 \}.$$

\end{proof}

\newpage

The structure of $K_l$ can be illustrated in the following picture:

\begin{figure}[h]
\center{\includegraphics[scale=0.42]{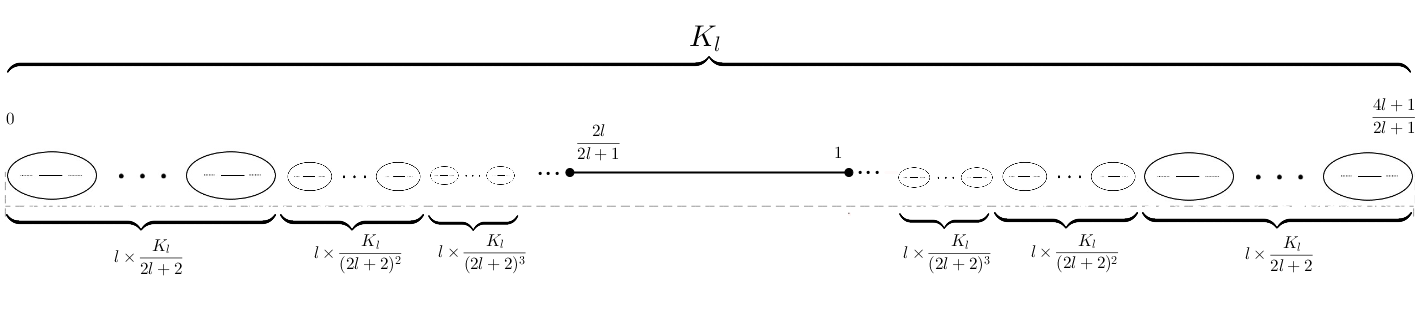}}
\end{figure}

\begin{alphthm}
\label{MT}
The Lebesgue measure of the Cantorval $K_l$ is equal to 1, while its boundary has a non-integer Hausdorff dimension equal to $~ \log(2l+1)/\log(2l+2)$.
\end{alphthm}

\begin{proof}

According to Lemmas \ref{Basic properties} and \ref{LM1}, $int(K_l)$ contains the interval $\left( \frac{2l}{2l+1}, 1\right)$. Iterating the decomposition from Lemma \ref{LM3}, we obtain $2l(2l+1)^{n-1}$ pairwise disjoint copies of this interval with ratio $1/(2l+2)^n$, for each $n \in \mathbb{N}$. These intervals form the whole interior of $K_l$, since any additional interval component would have to lie in nested outer copies of $K_l$ whose diameters tend to zero. Hence, its Lebesgue measure can be calculated as
$$\lambda(int(K_l))=\frac{1}{2l+1}+\frac{1}{2l+1} \cdot \frac{2l}{(2l+2)}+\frac{1}{2l+1} \cdot \frac{2l(2l+1)}{(2l+2)^2}+ \dots + \frac{1}{2l+1} \cdot \frac{2l(2l+1)^{n-1}}{(2l+2)^{n}}+\dots=$$
$$=\frac{1}{2l+1}+\frac{1}{2l+1} \cdot \frac{\frac{2l}{2l+2}}{1-\frac{2l+1}{2l+2}}=\frac{1}{2l+1}+\frac{2l}{2l+1}=1.$$

By Lemma \ref{LM3}, $fr(K_l)$ consists of a countable union of pairwise disjoint affine copies of itself up to a countable set of endpoints of the component intervals. Moreover, this union contains exactly $2l$ similar copies of this set with ratio $1/(2l+2)^n$ for each $n \in \mathbb{N}$.
Sets of this kind are commonly referred to as countable self-similar or $N$-self-similar (see \cite{Mor96} and \cite{Pra98}). Since even the convex hulls of the copies are pairwise disjoint, the associated countable IFS satisfies the OSC with the open set $\left(0,\frac{4l+1}{2l+1}\right)$. The at most countable exceptional set does not affect the Hausdorff dimension due to countable stability. Therefore, the Hausdorff dimension of $fr(K_l)$ coincides with the $N$-similarity dimension, which is the unique solution of the following equation:
$$2l \cdot \left( \frac{1}{2l+2} \right)^s + 2l \cdot \left( \frac{1}{2l+2} \right)^{2s} + \dots + 2l \cdot \left( \frac{1}{2l+2} \right)^{ns}+\dots= 1.$$
It follows that
$$\frac{(2l+2)^{-s}}{1-(2l+2)^{-s}}=\frac{1}{2l} \ \ \ \Rightarrow \ \ \ (2l+1)(2l+2)^{-s}=1 \ \ \ \Rightarrow \ \ \ s=\log(2l+1)/ \log(2l+2).$$

Since $\log(2l+1)/\log(2l+2)<1$, we have $\lambda(fr(K_l))=0$ and therefore $\lambda(K_l)=\lambda(int(K_l))=1$.
Hence, $int(K_l)$ is a countable union of open intervals with total length 1, while $fr(K_l)$ is a fractal set with Hausdorff dimension equal to $\log(2l+1)/\log(2l+2)$.
\end{proof}

\begin{corollary}
The Guthrie--Nymann Cantorval (the case $l=1$)
$$K_1=\left\{ \sum_{n=1}^{\infty} \frac{\varepsilon_n}{4^n} : (\varepsilon_n) \in \{0, 2, 3, 5\}^{\mathbb{N}} \right\}$$
has Lebesgue measure $1$, while its boundary is a fractal set whose Hausdorff dimension is equal to $\log{3}/\log{4}$.
\end{corollary}

\section{Final remarks and open questions}

It is easy to see that the similarity dimension of the IFS $\Psi_l=\{ \psi_i(x)\}_{i=1}^{2l+2}$ is equal to $1$. Since its attractor $K_l$ has positive Lebesgue measure, the IFS $\Psi_l$ satisfies the OSC by Theorem~\ref{T:Schief}. The interior of $K_l$ is a countable union of open intervals that satisfies:
\begin{enumerate}
\item $\psi_i(int(K_l)) \subset int(K_l)$,
\item $\psi_i(int(K_l)) \cap \psi_j(int(K_l)) = \emptyset, i \neq j,$
\end{enumerate}
and serves as the open set $V$ in the definition of the OSC. The example demonstrates that even in this seemingly simple case, when an IFS satisfies the OSC, its attractor can have an intricate structure.

We have established metric and fractal properties for a countable family of self-similar Cantorvals associated with multigeometric series.
A natural question arises: can analogous results be derived for other families of Cantorvals, for instance those considered in \cite{Ferdinands} or \cite{Bartoszewicz}? Another question that deserves special attention concerns the Hausdorff dimension of the boundaries of Cantorvals that arise as sets of subsums of numerical series. Is it possible for the boundary to have an integer fractal dimension? Can the boundary be a fat Cantor set, i.e., have positive Lebesgue measure?
In \cite{NP25}, Nowakowski and Prus-Wi\'sniowski prove that the boundaries of multigeometric Cantorvals along with a more general class of standard Cantorvals have Lebesgue measure zero. However, the problem remains open in a general framework.

\end{document}